\documentclass[12pt]{article}
\usepackage[T2A]{fontenc}
\usepackage[english]{babel}
\usepackage[tbtags]{amsmath}
\usepackage{amsfonts,amssymb}
\usepackage{amssymb}
\usepackage{graphicx}
\usepackage{amscd}

\usepackage{amscd}
\usepackage{color}
\usepackage{fullpage}
\usepackage{float}

\usepackage{graphics,amsmath,amssymb}
\usepackage{amsthm}
\usepackage{amsfonts}
\usepackage{latexsym}
\usepackage{epsf}
\usepackage{mathrsfs}
\usepackage{hyphenat}
\usepackage[colorlinks=true,
linkcolor=webgreen,
filecolor=webbrown,
citecolor=webgreen]{hyperref}

\definecolor{webgreen}{rgb}{0,.5,0}
\definecolor{webbrown}{rgb}{.6,0,0}

\usepackage{color}
\usepackage{fullpage}
\usepackage{float}

\usepackage{fullpage}
\usepackage{float}

\usepackage{graphics,amsmath,amssymb}
\usepackage{amsthm}
\usepackage{amsfonts}
\usepackage{mathrsfs}

\title{Method for solving an iterative functional equation $A^{2^n}(x)=F(x)$}

\author{D. V. Kruchinin, V.~V.~Kruchinin\\
\small Tomsk State University of Control Systems and Radioelectronics, Russian Federation\\
\small \texttt{kruchininDm@gmail.com}\\
}

\begin{document}
\maketitle

\begin{abstract}
Using the notion of the composita, we obtain a method of solving iterative functional equations of the form $A^{2^n}(x)=F(x)$, where $F(x)=\sum_{n>0} f(n)x^n$, $f(1)\neq 0$. 
We prove that if $F(x)=\sum_{n>0} f(n)x^n$ has integer coefficients, then the generating function $A(x)=\sum_{n>0}a(n)x^n$, which is obtained from the iterative functional equation $4A(A(x))=F(4x)$, has integer coefficients.

Key words: iterative functional equation, composition of generating functions, composita.
\end{abstract}

\theoremstyle{plain}
\newtheorem{theorem}{Theorem}
\newtheorem{Corollary}[theorem]{Corollary}
\newtheorem{lemma}[theorem]{Lemma}
\newtheorem{proposition}[theorem]{Proposition}

\theoremstyle{definition}
\newtheorem{definition}[theorem]{Definition}
\newtheorem{example}[theorem]{Example}
\newtheorem{conjecture}[theorem]{Conjecture}
\theoremstyle{remark}
\newtheorem{remark}[theorem]{Remark}

\newtheorem{Theorem}{Theorem}[section]
\newtheorem{Proposition}[Theorem]{Proposition}

\theoremstyle{definition}
\newtheorem{Example}[Theorem]{Example}
\newtheorem{Remark}[Theorem]{Remark}
\newtheorem{Problem}[Theorem]{Problem}
\newtheorem{state}[Theorem]{Statement}
\makeatletter
\def\rdots{\mathinner{\mkern1mu\raise\p@
\vbox{\kern7\p@\hbox{.}}\mkern2mu
\raise4\p@\hbox{.}\mkern2mu\raise7\p@\hbox{.}\mkern1mu}}

\newcommand{\seqnum}[1]{\href{http://oeis.org/#1}{\underline{#1}}}

\newcommand{\seqnumr}[1]{\begin{flushright}{\href{http://oeis.org/#1}{\underline{#1}}}\end{flushright}}

\makeatother

\section{Introduction}
For a given $A(x)$ the $n$-{th} iteration is the function that is composed
with itself $n$ times.
$$
A^0(x)=x, \qquad A^1(x)=A(x), \qquad A^2(x)=A(A^1(x)), \ldots, A^n(x)=A(A^{n-1}(x))
$$
and denoted by $A^n(x)$.

By an iterative functional equation  we mean the equation, where the analytical form of a function $A(x)$ is not known, but its
composition with itself is known.

Iterative functional equations  arise in various fields such as fractal theory, computer science, dynamical systems, and maps.
In particular, a challenging iterated functional equation $A(A(x))=x^2-2$ circulate as a test of one's intelligence.

Despite their prevalence, they are very difficult to solve, and few mathematical tools exist to analyse them.
In the general case, solving the equations involve deep mathematical insight
and experimentation with different substitutions and reformulations.

In the literature you can find extensive investigations concerning iterates and functional equations, classical books \cite{Aczel,Kuczma_1968,Kuczma_1990,Nechepurenko} and recent surveys and papers \cite{BaronJar,Kobza,BBS,Berg,Liu}.

In this paper, we propose using the notion of the \textit{composita} to identify and solve such iterative functional equations, where $F(x)=\sum_{n>0} f(n)x^n$, $f(1)\neq 0$.

\section{Preliminary}

In the paper \cite{KruCompositae}, the authors introduced the notion of the \textit{composita} of a given ordinary generating function $F(x)=\sum_{n>0}f(n)x^n$.

\begin{definition} 
The composita of the generating function $F(x)=\sum_{n>0}f(n)x^n$ is the function of two variables
\begin{equation}
\label{Fnk0}F^{\Delta}(n,k)=\sum_{\pi_k \in C_n}{f(\lambda_1)f(\lambda_2)\ldots f(\lambda_k)},
\end{equation}
where $C_n$ is a set of all compositions of an integer $n$, $\pi_k$ is the composition $\sum_{i=1}^k\lambda_i=n$ into $k$ parts exactly.
\end{definition}

The generating function of the composita of $F(x)$ is equal to
\begin{equation}
\label{GenComp}
[F(x)]^k=\sum_{n\geq k}F^{\Delta}(n,k)x^n.
\end{equation} 

The composita is the coefficients of the powers of an ordinary generating function
$$F^{\Delta}(n,k) := [z^{n}] F(x)^k.$$

For the case $k=n$ the composita $F^{\Delta}(n,k)$ (see the formula (\ref{Fnk0})  is
\begin{equation}\label{formula_f1n}
F^{\Delta}(n,n)=f(1)^n.
\end{equation}

In tabular form, the composita is presented as a triangle as follows
$$
\begin{array}{ccccccccccc}
&&&&& f(1)\\
&&&& f(2) && f^2(1)\\
&&& f(3) && F_{3,2}^{\Delta} && f^3(1)\\
&& f(4) && F_{4,2}^{\Delta} && F_{4,3}^{\Delta} && f^4(1)\\
& \rdots && \vdots && \vdots && \vdots && \ddots\\
f(n) && F_{n,2}^{\Delta} && \ldots && \ldots && F_{n,n-1}^{\Delta} && f^n(1)\\
\end{array}
$$ 

Comtet\cite[ p.\ 141]{Comtet_1974} considered similar objects and identities for exponential generating functions, and called them potential polynomials. In this paper we consider the more general case of generating functions ordinary generating functions.
For more information about the composita you can see \cite{KruCompositae,KruchininTriangle2012,KruApplication}.

In the following theorem we give the formula for the composita of a composition of generating function, which was proved by the authors in \cite{KruCompositae}.

\begin{theorem}\label{CompozitProduct}  
Suppose we have the generating functions $A(x)=\sum_{n>0}a(n)x^n$, $G(x)=\sum_{n>0} g(n)x^n$, and their compositae $A^{\Delta}(n,k)$, $G^{\Delta}(n,k)$ respectively.
Then for the composition of generating functions $F(x)=G\left(A(x)\right)$  the composita is equal to 
\begin{equation}
\label{compCompositon}
F^{\Delta}(n,k)=\sum_{m=k}^n A^{\Delta}(n,m)G^{\Delta}(m,k).
\end{equation}
\end{theorem}

\section{Iterative functional equation  $A(A(x))=F(x)$}

The simplest case of an iterative functional equation for an unknown function $A(x)$ with given
function $F(x)$
$$
A(A(x))=F(x)
$$
appeared in 1820 (Babbage, with $F(x)=x$).

In this section we consider the case, when the expression for coefficients of a generating function $A(x)=\sum_{n>0}a(n)x^n$ is not known, but its composition with itself is known.

\begin{theorem}\label{TheoremItaration2}
Suppose $F(x)=\sum_{n>0} f(n)x^n$ is a generating function, where $f(1)\neq 0$, $F^{\Delta}(n,k)$ is the composita of $F(x)$, $A(x)=\sum_{n>0}a(n)x^n$ is the generating function, which is obtained from the functional equation $A(A(x))=F(x)$. Then for the composita of the generating function $A(x)$ we have the following recurrent formula
\begin{equation}\label{FormulaIteration}
A^{\Delta}(n,k)=\left\{
\begin{array}{ll}
{f(1)}^{\frac{n}{2}}, & n=k;\\
\frac{1}{{f(1)}^{n\over 2}+{f(1)}^{k\over 2}}\left(F^{\Delta}(n,k)-\sum\limits_{m=k+1}^{n-1} A^{\Delta}(n,m)A^{\Delta}(m,k)\right),& n>k.\\
\end{array}
\right.
\end{equation}
\end{theorem}

\begin{proof}
Using Theorem \ref{CompozitProduct}, for the case $n=k$ we get
$$
 A^{\Delta}(n,n)A^{\Delta}(n,n)=F^{\Delta}(n,n).
$$
According to the formula (\ref{formula_f1n}), we obtain
$$
A^{\Delta}(n,n)=f(1)^{\frac{n}{2}}.
$$

Using Theorem \ref{CompozitProduct}, for the case $n>k$  we get
$$
\sum_{m=k}^n A^{\Delta}(n,m)A^{\Delta}(m,k)=F^{\Delta}(n,k)
$$
or
$$
A^{\Delta}(n,k)A^{\Delta}(k,k)+\sum_{m=k+1}^{n-1}A^{\Delta}(n,m)A^{\Delta}(m,k)+A^{\Delta}(n,n)A^{\Delta}(n,k)=F^{\Delta}(n,k).
$$

Therefore, we have
$$
A^{\Delta}(n,k)\left(A^{\Delta}(k,k)+A^{\Delta}(n,n)\right)=F^{\Delta}(n,k)-\sum_{m=k+1}^{n-1}A^{\Delta}(n,m)A^{\Delta}(m,k).
$$
Since $A^{\Delta}(n,n)=a(1)^n=f(1)^{\frac{n}{2}}$, we obtain the desired formula
$$
A^{\Delta}(n,k)=\left\{
\begin{array}{ll}
{f(1)}^{\frac{n}{2}}, & n=k;\\
\frac{1}{{f(1)}^{n\over 2}+{f(1)}^{k\over 2}}\left(F^{\Delta}(n,k)-\sum\limits_{m=k+1}^{n-1} A^{\Delta}(n,m)A^{\Delta}(m,k)\right),& n>k.\\
\end{array}
\right.
$$

Next we show that the formula (\ref{FormulaIteration}) is computable.
For the case $n=1$ the formula is computable
$$
A^{\Delta}(1,1)=\sqrt{f(1)}.
$$
By induction, we put that there exist $A^{\Delta}(k,l)$ for $k,l\leqslant n$.

Then we prove that  $A^{\Delta}(n+1,m)$ is computable for $m\leqslant n+1$.
\begin{enumerate}
\item First case $m=n+1$. From Theorem \ref{CompozitProduct}, we have 
$$A^{\Delta}(n+1,n+1)A^{\Delta}(n+1,n+1)=F^{\Delta}(n+1,n+1).$$
Then $A^{\Delta}(n+1,n+1)$ is computable.

\item Next case $m=n$. From Theorem \ref{CompozitProduct}, we have 
$$A^{\Delta}(n+1,n)A^{\Delta}(n,n)+A^{\Delta}(n+1,n+1)A^{\Delta}(n+1,n)=F^{\Delta}(n+1,n).$$
Since $A^{\Delta}(n,n)$ and $A^{\Delta}(n+1,n+1)$ are known, $A^{\Delta}(n+1,n)$ is computable.

\item The case $m=n-1$. From Theorem \ref{CompozitProduct}, we have 
$$A^{\Delta}(n+1,n-1)A^{\Delta}(n+1,n+1)+A^{\Delta}(n,n-1)A^{\Delta}(n+1,n)+$$
$$+A^{\Delta}(n-1,n-1)A^{\Delta}(n+1,n-1)=F^{\Delta}(n+1,n).$$
Since $A^{\Delta}(n+1,n+1)$,$A^{\Delta}(n+1,n)$ are obtained previously, and other are known by the condition, $A^{\Delta}(n+1,n-1)$ is computable.

Therefore, by induction we can obtain all values of  $A^{\Delta}(n+1,m)$ for $m\leqslant n+1$.
\end{enumerate} 
\end{proof}

For applications of Theorem~\ref{TheoremItaration2} we give some examples.
\begin{example} Let us consider the iterative functional equation $A(A(x))=\sin(x)$ (for details, see \seqnum{A048602}, \seqnum{A048603} in \cite{oeis}).
According to \cite{KruCompositae}, the composita of $\sin(x)$ is equal to
$$
\frac{1+(-1)^{n-k}}{2^{k}n!}\sum_{m=0}^{\lfloor \frac{k}{2} \rfloor} {k \choose m} (2m-k)^n(-1)^{\frac{n+k}{2}-m}.
$$
Using Theorem \ref{TheoremItaration2}, we obtain following recurrent formula
$$ A^{\Delta}(n,k)=\left\{
\begin{array}{ll}
1, & n=k;\\
\frac{1}{2}\left({{\frac{1+(-1)^{n-k}}{2^{k}n!}\sum\limits_{m=0}^{\lfloor \frac{k}{2} \rfloor} {k \choose m} (2m-k)^n(-1)^{\frac{n+k}{2}-m}-\sum\limits_{m=k+1}^{n-1}A^{\Delta}(n,m)A^{\Delta}(m,k)}}\right), & n>k.\\
\end{array}
\right.
$$ 

Therefore, using the formula (\ref{GenComp}), the expression for coefficients of $A(x)=\sum_{n>0}a(n)x^n$ is
$$ a(n)=A^{\Delta}(n,1).
$$ 
\end{example}

\begin{example}
 Let us consider the iterative functional equation $A(A(x))=e^x-1$ (for details, see \seqnum{A052104}, \seqnum{A052122}, and \seqnum{A052105} in \cite{oeis}).
Stanley studied the problem of finding an expression for the coefficients of $A(x)$, however the problem remains unsolved \cite[ ex.\ 5.52.c]{Stanley_v2}.

Next, we obtain the composita of the exponential generating function \cite{KruCompositae}
$$
F(x)=e^x-1.
$$

Raising this generating function to the power of $k$ and applying the binomial theorem, we obtain 
$$
\left(e^x-1\right)^k=\sum_{j=0}^k {k \choose j}e^{xj}(-1)^{k-j}.
$$
Since 
$$
\left(e^{x} \right)^k=\sum_{n\geq 0} \frac{k^n}{n!}x^n,
$$
we get
$$
F^{\Delta}(n,k)=\sum_{j=0}^k {k \choose j}\frac{j^n}{n!}(-1)^{k-j}.
$$
or
since the general formula for the Stirling numbers of the second kind is given as follows:
$$
\genfrac{\{}{\}}{0pt}{}{n}{k}=\frac{1}{k!}\sum_{j=0}^k(-1)^{k-j}\binom{k}{j}j^n,
$$
we have
\begin{equation}
\label{exp(x)-1}
F^{\Delta}(n,k)=\frac{k!}{n!}\genfrac{\{}{\}}{0pt}{}{n}{k}.
\end{equation}

Here 
$
\genfrac{\{}{\}}{0pt}{}{n}{k}
$
stand for the Stirling numbers  of the second kind(see \cite{Comtet_1974,ConcreteMath}).
The Stirling numbers of the second kind 
$
\genfrac{\{}{\}}{0pt}{}{n}{k}=S(n,k)
$
count the number of ways to partition a set of $n$ elements into $k$ nonempty subsets.

Using Theorem \ref{TheoremItaration2}, we obtain following recurrent formula
$$ A^{\Delta}(n,k)=\left\{
\begin{array}{ll}
1, & n=m,\\
\frac{1}{2}\left({{{{n!}\over{n!}}\,\genfrac{\{}{\}}{0pt}{}{n}{k}-\sum\limits_{m=k+1}^{n-1} A^{\Delta}(n,m)A^{\Delta}(m,k)}}\right), & n>k.\\
 \end{array}
 \right.
$$ 

Therefore, using the formula (\ref{GenComp}), the expression for coefficients of $A(x)=\sum_{n>0}a(n)\frac{x^n}{n!}$ is
$$ a(n)=n!\,A^{\Delta}(n,1).
$$ 
\end{example}

Other examples of  solutions of iterative functional equations  written in the On-line Encyclopedia of Integer Sequences (see Table  \ref{IterationTable}).

\begin{center}

\begin{table}\caption{Iterative functional equations  and their solutions in OEIS}\label{IterationTable}
\begin{center}
\begin{tabular}{|l|l|} \hline
N of sequence & Equation \\ \hline
\seqnum{A030274}   &    $A(A(x)) = \frac{x}{(1-x)^2}$\\ \hline
\seqnum{A097088}   &    $A(A(x)) = x+x^2$\\ \hline
\seqnum{A097090}   &	$A(A(x)) = x(1+2x)^2$\\ \hline
\seqnum{A048607}   &   $A(A(x)) = \ln(1+x)$\\ \hline
\seqnum{A072350}   &   $A(A(x)) = \tan(x)$\\ \hline
\seqnum{A199822}   &   $A(A(x))=\frac{1-4x-\sqrt{1-8x}}{8x}$\\ \hline
\seqnum{A199823}   &   $A(A(x)) = \frac{x+2x^2}{1-2x-4x^2}$\\ \hline
\seqnum{A199852}   &   $A(A(x))=x\exp(2x)$\\ \hline
\seqnum{A199917}   &   $A(A(x)) = \frac{1}{x}(2-2\cos(x))$\\ \hline
\end{tabular}
\end{center}
\end{table}
\end{center}

\section{Integer properties of an iterative functional equation  $A(A(x))=F(x)$}

In this section we study an iterative functional equation  $A(A(x))=F(x)$, where $F(x)=\sum_{n>0} f(n)x^n$ is a generating function with integer coefficients.

First for case $f(1)=1$, the formula (\ref{FormulaIteration}) has form
\begin{equation}\label{FormulaIteration1}
A^{\Delta}(n,k)=\left\{
\begin{array}{ll}
1, & n=k;\\
\frac{1}{2}\left(F^{\Delta}(n,k)-\sum\limits_{m=k+1}^{n-1} A^{\Delta}(n,m)A^{\Delta}(m,k)\right),& n>k.\\
\end{array}
\right.
\end{equation}

According to the above formula (\ref{FormulaIteration1}), we present the composita of $A(x)$ in tabular form

\begin{equation}\label{trian1}
\begin{array}{llll}
1\\
\frac{F_{2,1}^{\Delta}}{2}&1 \\
-\frac{(F_{2,1}^{\Delta}F_{3,2}^{\Delta}-4F_{3,1}^{\Delta})}{8}&\frac{F_{3,2}^{\Delta}}{2}&1\\
\frac{((F_{2,1}^{\Delta}F_{3,2}^{\Delta}-2F_{3,1}^{\Delta})F_{4,3}^{\Delta}-2F_{2,1}^{\Delta}F_{4,2}^{\Delta}+8F_{4,1}^{\Delta})}{16}&\frac{-(F_{3,2}^{\Delta}F_{4,3}^{\Delta}-4F_{4,2}^{\Delta})}{8}&\frac{F_{4,3}^{\Delta}}{2}&1
\end{array}
\end{equation}

Considering Triangle \ref{trian1}, we can see that not all expressions $2^{n-k}A^{\Delta}(n,k)$ is integer (for example for $A^{\Delta}(3,1)$).
Now we give the following theorem.
\begin{theorem}
\label{integer}
Suppose $F(x)=\sum_{n>0} f(n)x^n$ is a generating function with integer coefficients such that $f(1)=1$, $F^{\Delta}(n,k)$ is the composita of $F(x)$, $A(x)=\sum_{n>0}a(n)x^n$ is the generating function, which is obtained from the functional equation $A(A(x))=F(x)$. Then for the composita of the generating function $A(x)$ the expression
$$4^{n-k}A^{\Delta}(n,k)$$
is integer.
\end{theorem}  

\begin{proof} 
To prove the theorem we use the formula (\ref{FormulaIteration1}).

For the case $n<5$ the expression
$4^{n-k}A^{\Delta}(n,k)$
is integer for all $k\leqslant n$ (see Triangle \ref{trian1}).

By induction, we put that $4^{n-k}A^{\Delta}(k,l)$ is integer for $k,l\leqslant n$.

Then we prove that 
$$
4^{n+1-k}A^{\Delta}(n+1,k).
$$
is integer for all $k\leqslant n+1$.
 
According to the formula (\ref{FormulaIteration1}), we have
$$
4^{n+1-k}A^{\Delta}(n+1,k)=
$$
$$
=4^{n+1-k}\frac{1}{2}\left(F^{\Delta}(n+1,k)-\sum\limits_{m=k+1}^{n} A^{\Delta}(n+1,m)A^{\Delta}(m,k)\right)=
$$

\begin{equation}\label{t41}
=\frac{1}{2}\left(4^{n+1-k}F^{\Delta}(n+1,k)-\sum\limits_{m=k+1}^{n} 4^{n+1-m}A^{\Delta}(n+1,m)4^{m-k}A^{\Delta}(m,k)\right).
\end{equation}

Since $
4^{n+1-m}F^{\Delta}(n+1,m)
$ is even and 
$
4^{k-m}A^{\Delta}(k,m)
$ is integer, we need to prove that the expression
$$
 4^{n+1-m}A^{\Delta}(n+1,m)
$$
is even for $m<n+1$. 

\begin{enumerate}
\item First case $m=n$. 
According to the formula (\ref{t41}), we have
$$
4A^{\Delta}(n+1,n)=4F^{\Delta}(n+1,n)\frac{1}{2}.
$$
Then $4^{n+1-m}A^{\Delta}(n+1,m)$ is even for $m=n$.

\item Next case $m=n-1$. 
According to the formula (\ref{t41}), we have
$$
4^{2}A^{\Delta}(n+1,n-1)=\frac{1}{2}\left(4^{2}F^{\Delta}(n+1,n-1)-4A^{\Delta}(n,n-1)4A^{\Delta}(n+1,n)\right).
$$
Since above conditions, the expression $4^{n+1-m}A^{\Delta}(n+1,m)$ is even for $m=n-1$.

\item The case $m=n-2$. 
According to the formula (\ref{t41}), we have
$$
4^{3}A^{\Delta}(n+1,n-2)=\frac{1}{2}\left(4^{3}F^{\Delta}(n+1,n-2)-4^2A^{\Delta}(n,n-2)4^1A^{\Delta}(n+1,n)+\right.
$$
$$
\left.+4^2A^{\Delta}(n,n-2)4^1A^{\Delta}(n+1,n)+4^1A^{\Delta}(n-1,n-2)4^2A^{\Delta}(n+1,n-1)\right).
$$

Since $4^{2}A^{\Delta}(n+1,n-1)$ is even, and other are known by the condition, the expression   $4^{n+1-m}A^{\Delta}(n+1,m)$ is even for $m=n-2$.
\end{enumerate}

Therefore, by induction we can obtain all values of  $A^{\Delta}(n+1,m)$ for $m<n+1$ which will be even.

The theorem is proved.
\end{proof}

\begin{Corollary}  
Suppose $F(x)=\sum_{n>0} f(n)x^n$ is a generating function with integer coefficients such that $f(1)=1$, $F^{\Delta}(n,k)$ is the composita of $F(x)$. Then the generating function $A(x)=\sum_{n>0}a(n)x^n$, which is obtained from the functional equation $A(A(x))=\frac{F(4x)}{4}$, has integer coefficients.
\end{Corollary}

\begin{proof}

According to \cite{KruCompositae}, the composita of $\frac{F(4x)}{4}$ is equal to
$$4^{n-k}F^{\Delta}(n,k).$$

Next we consider the formula (\ref{FormulaIteration1}) 
$$
A^{\Delta}(n,k)=\left\{
\begin{array}{ll}
1, & n=k;\\
\frac{1}{2}\left(4^{n-k}F^{\Delta}(n,k)-\sum\limits_{m=k+1}^{n-1} A^{\Delta}(n,m)A^{\Delta}(m,k)\right),& n>k.\\
\end{array}
\right.
$$

Replacing $A^{\Delta}(n,k)$ by the composita $4^{n-k}B^{\Delta}(n,k)$, we get
$$
4^{n-k}B^{\Delta}(n,k)=\frac{1}{2}\left(4^{n-k}F^{\Delta}(n,k)-\sum\limits_{m=k+1}^{n-1} 4^{n-m}B^{\Delta}(n,m)4^{m-k}B^{\Delta}(m,k)\right).
$$

According to Theorem \ref{integer}, the expression $4^{n-k}B^{\Delta}(n,k)$ is integer.

Therefore, the generating function $A(x)$ such that satisfies to the equation $$A(A(x))=\frac{F(4x)}{4},$$ where $F(x)=\sum_{n>0} f(n)x^n$ is a generating function with integer coefficients and $f(1)=1$, has integer coefficients.

The corollary is proved.
\end{proof}

\begin{example}
Let us consider the iterative functional equation 
$$A\left(A\left(x\right)\right)={{1-\sqrt{1-16\,x}}\over{8}}.$$ 
Also the generating function $A(x)$ satisfies the following functional equation (for details, see \seqnum{A213422} in \cite{oeis})
$$
A(A(x)-4A(x)^2)=x, \qquad A(A(x)) = x + 4A(A(x))^2.
$$
Here $F(x)$ is the generating function for the Catalan numbers
$$
F(x)={{1-\sqrt{1-4\,x}}\over{2}}.
$$
According to \cite{KruCompositae}, the composita of $\frac{F(4x)}{4}={{1-\sqrt{1-16\,x}}\over{8}}$ is equal to
$$
4^{n-k}F^{\Delta}(n,k)=4^{n-k}{{2n-k-1}\choose{n-1}}\frac{k}{n}.
$$

Using Theorem \ref{TheoremItaration2}, we obtain following recurrent formula

$$ A^{\Delta}(n,k)=\left\{
\begin{array}{ll}
1, & n=k;\\
\frac{1}{2}\left({{4^{n-k}\,{{2\,n-k-1}\choose{n-1}}}\frac{k}{n}}-\sum\limits_{k=m-1}^{n-1} A^{\Delta}(n,k)A^{\Delta}(k,m)\right),& n>k.\\
\end{array}
\right.
$$ 
According to the above formula, we present the composita of $A(x)$ in tabular form
$$
\begin{array}{llll}
1,\\
2F^{\Delta}_{2,1}, &1 \\
8F^{\Delta}{3,1}-2F^{\Delta}_{2,1}A_{3,2},&2F^{\Delta}_{3,2},&1\\
(4F^{\Delta}_{2,1}F^{\Delta}_{3,2}-8F^{\Delta}_{3,1})F^{\Delta}_{4,3}-8F^{\Delta}_{2,1}F^{\Delta}_{4,2}+32F^{\Delta}_{4,1},& 8F^{\Delta}_{4,2}-2F^{\Delta}_{3,2}F^{\Delta}_{4,3},&2F^{\Delta}_{4,3},&1\\
\end{array}
$$
Therefore, using the formula (\ref{GenComp}), the expression for coefficients of $A(x)=\sum_{n>0}a(n)x^n$ is
$$ a(n)=A^{\Delta}(n,1).
$$ 
The first elements of the obtained generating function $A(x)$ are shown below
$$x+2x^2+12x^3+96x^4+880x^5+8720x^6+90752x^7+\cdots
$$
\end{example}

\section{Iterative functional equation $A^{2^n}(x)=F(x)$}
We start with the case $A(A(A(A(x))))=F(x)$.
Let us get the following iterative functional equation 
$$
B(B(x))=F(x),
$$
where $B(x)=A(A(x)).$

Using Theorem \ref{TheoremItaration2}, we obtain the composita of $B(x)$
$$
B^{\Delta}(n,k)=\left\{
\begin{array}{ll}
{f(1)}^{\frac{n}{2}}, & n=k;\\
\frac{1}{{f(1)}^{n\over 2}+{f(1)}^{k\over 2}}\left(F^{\Delta}(n,k)-\sum\limits_{m=k+1}^{n-1} B^{\Delta}(n,m)B^{\Delta}(m,k)\right),& n>k.\\
\end{array}
\right.
$$

Then using the formula (\ref{FormulaIteration}) again for the equation $B(x)=A(A(x))$, we obtain the composita $A^{\Delta}(n,k)$.

Let us consider this approach in the following example.
\begin{example}
Suppose we have the following iterative functional equation 
$$
A(A(A(A(x))))=x+16x^2.
$$
According to \cite{KruCompositae}, the composita of $F(x)=x+16x^2$ is equal to
$$
F^{\Delta}(n,k)={k \choose n-k}16^{n-k}.
$$
For the generating function $B(x)$ such that $B(B(x))=x+16x^2$ the composita is equal to
$$
B^{\Delta}(n,k)=\left\{
\begin{array}{ll}
1, & n=k;\\
\frac{1}{2}\left[{{{{k}\choose{n-k}}\,16^{n-k}-\sum\limits_{i=k+1}^{n-1}{B^{\Delta}\left(i
  , k\right)\,B^{\Delta}\left(n , i\right)}}}\right], & n>k.\\
   \end{array}
   \right.
$$
Since $B(x)=A(A(x))$, we obtain the composita of $A(x)$
$$
A^{\Delta}(n,k)=\left\{
\begin{array}{ll}
1, & n=k;\\
\frac{1}{2}\left[{{B^{\Delta}(n,k)-\sum\limits_{i=k+1}^{n-1}{A^{\Delta}\left(i
  , k\right)\,A^{\Delta}\left(n , i\right)}}}\right], & n>k.\\
   \end{array}
   \right.
$$
Therefore, using the formula (\ref{GenComp}), the expression for coefficients of $A(x)=\sum_{n>0}a(n)x^n$ is
$$ a(n)=A^{\Delta}(n,1).
$$ 
The first elements of the obtained generating function $A(x)$ are shown below
$$x+4x^2-48x^3+960x^4-23296x^5+616448x^6-16830464x^7+\cdots
$$
For details, see \seqnum{A141119} in \cite{oeis}.
\end{example}

Generalize the above proposed  approach for the iterative functional equation $A^{2^n}(x)=F(x)$.
\begin{enumerate}
\item Make substitution $B(x)=A^{(2^{n-1})}(x)$ and solve $B(B(x))=F(x)$. Then we obtain $B^{\Delta}(n,k)$.
\item Next make substitution $C(x)=B^{(2^{n-2})}(x)$ and solve $C(C(x))=B(x)$ with respect to $C(x)$. Then we obtain $C^{\Delta}(n,k)$.
\item Making the analogic substitution until $n=1$, we obtain the desired solution of $A(x)$.
\end{enumerate}

\end{document}